\documentclass[11pt,a4paper]{amsart}
\usepackage{enumerate}
\usepackage{amssymb}
\usepackage{mathtools}

\usepackage{cancel,xspace}
\usepackage{epsfig,fancybox}
\usepackage{graphicx,mathrsfs,ifpdf}
\usepackage{setspace}
\usepackage{xcolor}
\usepackage{cleveref}
\usepackage{todonotes}

\newtheorem*{theorem*}{Theorem}
\newtheorem{theorem}{Theorem}
\newtheorem{lemma}[theorem]{Lemma}

\theoremstyle{definition}

\renewcommand{\theclaim}{\arabic{claim}}

\newtheorem*{claim*}{Claim}
\newenvironment{proofofclaim*}[1][\proofname\ of the claim]{%
  \proof[#1]%
  
}{\endproof}

\numberwithin{equation}{section}
\numberwithin{figure}{section}

\newcommand{\RP}{\mathbb R \textup P}
\newcommand{\RR}{\mathbb R}
\newcommand{\ZZ}{\mathbb Z}

\DeclareMathOperator{\sd}{sd}

\DeclareMathOperator{\lk}{lk}
\DeclareMathOperator{\st}{st}

\begin{document}
\title{Fan's lemma via bistellar moves}
\author{Tom\'a\v s Kaiser}
%\thanks{The first author was supported by project GA20-09525S of the Czech
%Science Foundation.}
\address{Department of Mathematics and European Centre of Excellence NTIS (New
Technologies for the Information Society), University of West Bohemia, Pilsen,
Czech Republic.}
\email{kaisert@kma.zcu.cz}
\author{Mat\v ej Stehl\'ik}
\thanks{The second author was partially supported by ANR project HOSIGRA
(ANR-17-CE40-0022).}
\address{Université Paris Cité, CNRS, IRIF, F-75006, Paris, France.}
\email{matej@irif.fr}
\date{\today}

\begin{abstract}
Pachner proved that all closed combinatorially equivalent combinatorial
manifolds can be transformed into each other by a finite sequence of bistellar
moves. We prove an analogue of Pachner's theorem for combinatorial manifolds
with a free $\ZZ_2$-action, and use it to give a combinatorial proof of Fan's
lemma about labellings of centrally symmetric triangulations of spheres.
Similarly to other combinatorial proofs, we must assume an additional property
of the triangulation for the proof to work. However, unlike the other
combinatorial proofs, no such assumption is needed for dimensions at most $3$.
\end{abstract}
\maketitle

\section{Introduction}
\label{sec:introduction}

Given a simplicial complex $K$ and a labelling $\lambda : V(K) \to \ZZ\setminus
\{0\}$, an $n$-simplex of $K$ is \emph{alternating} (with respect to $\lambda$)
if the $n+1$ vertices are assigned distinct labels, and their signs alternate by
increasing absolute value. An alternating simplex is \emph{positive} or
\emph{negative} if the label of minimum absolute value is positive or negative,
respectively. The following result, proved by Fan~\cite{Fan52} in 1952, is
usually referred to in the literature as \emph{Fan's lemma} (see
also~\cite{DLGMM19,dL13}).

\begin{theorem}[Fan]
\label{thm:fan}
Let $K$ be a centrally symmetric triangulation of $S^n$, and $\lambda: V(K) \to
\{\pm 1,\ldots,\pm m\}$ a labelling of the vertices of $K$ such that
$\lambda(-v) = -\lambda(v)$ for every vertex $v \in V(K)$, and
$\lambda(u)+\lambda(v) \neq 0$ for every edge $\{u,v\} \in K$. Then there are an
odd number of positive alternating $n$-simplices. In particular, $m \geq n+1$.
\end{theorem}

The contraposition, known as \emph{Tucker's lemma}, was originally proved by
Tucker~\cite{Tuc46} for dimension $2$ and subsequently by Lefschetz~\cite{Lef49}
for arbitrary dimension (see also~\cite{DLGMM19,dL13,Mat03}).

\begin{theorem}[Tucker]
\label{thm:tucker}
Let $K$ be a centrally symmetric triangulation of $S^n$, and $\lambda: V(K) \to
\{\pm 1,\ldots,\pm n\}$ a labelling of the vertices of $K$ such that
$\lambda(-v) = -\lambda(v)$ for every vertex $v \in V(K)$. Then there exists an
edge $\{u,v\} \in K$ such that $\lambda(u)+\lambda(v)=0$.
\end{theorem}

Tucker's lemma can be viewed as a discrete analogue of the following theorem of
Borsuk~\cite{Bor33}, usually known as the \emph{Borsuk--Ulam theorem}.

\begin{theorem}[Borsuk]
There is no continuous mapping $f:S^n \to S^{n-1}$ such that $f(-x)=-f(x)$ for
all $x \in S^n$.
\end{theorem}

The Borsuk--Ulam theorem has many important applications in combinatorics, so
much so that Matoušek wrote an entire book on the topic~\cite{Mat03}. Since the
Borsuk--Ulam theorem is implied by Tucker's (and, \emph{a fortiori}, by Fan's)
lemma, it ought to be possible---at least in principle---to `discretise' proofs
relying on the Borsuk--Ulam theorem by applying the lemmas of Tucker or Fan
instead (examples of this can be found, for example,
in~\cite{Mat04,Meu08,MS19,Zie02}). To obtain a truly `combinatorial' proof, one
then needs a combinatorial proof of the lemmas of Tucker and Fan.

A number of such proofs are known, but none work in full generality: they all
impose some extra condition on the triangulation (for a short degree-theoretic
proof of Fan's lemma with no extra conditions on the triangulation,
see~\cite{DLGMM19}). Several proofs require the triangulation to be a
subdivision of the boundary complex of the cross polytope; this is the case for
the original proof of Tucker's lemma given by
Lefschetz~\cite[pp.~135--137]{Lef49} and the original proof of Fan's lemma by
Fan~\cite{Fan52}. It is also required in the constructive proof of Tucker's
lemma by Freund and Todd~\cite{FT81} (reproduced in Matou\v
sek~\cite[pp.~37--40]{Mat03}), and in the proofs of Fan's lemma given by
Meunier~\cite{Meu08} and de Longueville~\cite[Theorem~1.8]{dL13}. Prescott and
Su~\cite{PS05} impose a less restrictive condition in their proof of Fan's
lemma: they require the triangulation to contain what they call a \emph{flag of
hemispheres}. The second proof of Tucker's lemma in Matou\v
sek~\cite[pp.~43--45]{Mat03} also requires this condition.

Can we do away with these extra conditions on the triangulation? Indeed,
\cite[Open Problem~2.4]{DLGMM19} asks for a direct combinatorial proof of Fan’s
and Tucker’s lemma valid for \emph{any} centrally symmetric triangulation of
$S^n$.

In this note, we give a new combinatorial proof of Fan's lemma using
\emph{bistellar moves} (also known as \emph{bistellar flips} or \emph{Pachner
moves}). The proof also imposes some extra conditions on the triangulation (see
\Cref{thm:fan-PL}), however, this only matters for dimensions $4$ and above.
Therefore, in dimensions $2$ and $3$, our proof establishes Fan's lemma for
\emph{all} centrally symmetric triangulations of the sphere (see the discussion
in \Cref{sec:conclusion}).

The rest of this paper is organised as follows. In the next section, after
introducing the basic terminology about simplicial complexes and combinatorial
manifolds, we introduce bistellar moves and state the fundamental theorem of
Pachner~\cite{Pac91} and its extension by Casali~\cite{Cas95}, and conclude the
section by proving a $\ZZ_2$-analogue of Pachner's theorem. In \Cref{sec:proof},
we prove Fan's lemma using bistellar moves. Finally, in \Cref{sec:conclusion},
we discuss the conditions our proof imposes on the triangulation.

\section{Terminology and preliminary results}
\label{sec:terminology}

For an introduction to topological combinatorics, we refer the reader to
Matou\v{s}ek~\cite{Mat03}, de Longueville~\cite{dL13} or Kozlov~\cite{Koz08}.

All the simplicial complexes considered in this paper are finite. The union of
all simplices in a simplicial complex $K$ is denoted by $|K|$ and is known as
the \emph{polyhedron} of $K$. A simplicial complex $K$ is a (simplicial)
\emph{triangulation} of any topological space homeomorphic to $|K|$.

By listing the vertices spanned by each face of a simplicial complex, we can
view simplicial complexes in a purely combinatorial way. An (abstract)
\emph{simplicial complex} $K$ on a (finite) vertex set $V = V(K)$ is a
collection of subsets $A \subseteq V(K)$ called \emph{faces} with the property
that if $A \in K$ and $B \subseteq A$, then $B \in K$. Any abstract simplicial
object has a geometric realization in Euclidean space, which is unique up to
homeomorphism. In the rest of this paper, simplicial complexes will be formally
treated as abstract simplicial complexes.

The \emph{dimension} of a face $A \in K$ is $|A|-1$, and the dimension of a
simplicial complex $K$ is the maximum dimension of a face of $K$. A \emph{facet}
of $K$ is an inclusion-wise maximal face. If $K$ and $L$ are simplicial
complexes on disjoint vertex sets, their \emph{join} is the simplicial complex
\[
K \star L = \{A \cup B : A \in K \text{ and } B \in L\}.
\]
The \emph{barycentric subdivision} of a simplicial complex $K$ is denoted by
$\sd(K)$.
For a given face $A \in K$, the \emph{link} $\lk_K(A)$ is defined by
\[
\lk_K(A) = \{B \in K: A \cap B = \emptyset \text{ and } A \cup B \in K\}.
\]
For a vertex $v$ in $K$, we abbreviate $\lk_K(\{v\})$ by $\lk_K(v)$. The
\emph{boundary} $\partial A$ of an $n$-simplex $A$ is the simplicial complex
consisting of all the faces of $A$ of dimension smaller than $n$. The (closed)
\emph{star} of $A$ in $K$, $\st_K(A)$, is the join $A\star \lk_K(A)$. The
operation $(A, a)$, called a \emph{stellar subdivision}, changes $K$ to $K'$ by
removing $\st_K(A)$ and replacing it with $a \star \partial A \star \lk_K(A)$,
where $a$ is a new vertex.

We will sometimes need to consider a slight generalisation of simplicial
complexes called \emph{regular $\Delta$-complexes}~\cite{Hat02} or
\emph{generalised simplicial complexes}~\cite{Koz08}. A regular $\Delta$-complex
can be constructed inductively, starting with a discrete point space $K^{(0)}$
in $\RR^n$, and at each step $i > 0$ we inductively construct the space
$K^{(i)}$ by attaching a set of $i$-simplices to $K^{(i-1)}$. Note that this
definition allows for several simplices on the same vertex set, which cannot
happen in a simplicial complex. It is however not hard to see that if $K$ is a
regular $\Delta$-complex, then $\sd(K)$ is a simplicial complex.

Two simplicial complexes $K$ and $L$ are \emph{combinatorially equivalent} (or
\emph{piecewise linearly homeomorphic}) if there are a subdivision $K'$ of $K$
and a subdivision $L'$ of $L$ such that $K'$ and $L'$ are isomorphic. A
\emph{combinatorial $n$-ball} is a simplicial complex that is combinatorially
equivalent to the $n$-simplex; a \emph{combinatorial $n$-sphere} is a simplicial
complex that is combinatorially equivalent to the boundary of an
$(n+1)$-simplex. A \emph{combinatorial $n$-manifold} (also known as a
\emph{piecewise linear $n$-manifold}) is a simplicial complex $M$ such that the
link of every vertex is a combinatorial $(n-1)$-sphere or a combinatorial
$(n-1)$-ball (the latter occurring only if the manifold has a non-empty
boundary). An $(n-1)$-face of a combinatorial $n$-manifold $M$ is a
\emph{boundary face} if it belongs to a single facet of $M$. The \emph{boundary
complex} of a combinatorial manifold $M$, denoted by $\partial M$, is the
subcomplex of $M$ consisting of all boundary facets of $M$ together with their
faces.

Suppose that $A$ is an $r$-simplex in a simplicial complex $K$ and that
$\lk_K(A)=\partial B$ for some $(n-r)$-simplex $B \notin K$. The \emph{bistellar
move} $\kappa(A, B)$ consists of changing $K$ by replacing $A \star \partial B$
with $\partial A \star B$. Two combinatorial manifolds are said to be
\emph{bistellar equivalent} if one can be obtained from the other via a sequence
of bistellar moves and simplicial isomorphisms. The following fundamental result
was proved by Pachner~\cite{Pac91} (see Lickorish~\cite{Lic99} for a very
readable exposition).

\begin{theorem}[Pachner]
\label{thm:pachner}
Two closed combinatorial manifolds are combinatorially equivalent if and only if
they are bistellar equivalent.
\end{theorem}

Casali~\cite{Cas95} extended \Cref{thm:pachner} to combinatorial manifolds with
a boundary. Note that in the theorem below, the boundary is unaffected by
bistellar moves, which act only on the facets.

\begin{theorem}[Casali]
\label{thm:casali}
Two combinatorial manifolds with isomorphic boundaries are combinatorially
equivalent if and only if they are bistellar equivalent.
\end{theorem}

Since Fan's lemma deals with centrally symmetric triangulations of spheres, we
need to adapt combinatorial and bistellar equivalence to the symmetric setting.
A \emph{simplicial $\ZZ_2$-complex} is a simplicial complex $K$ together with a
simplicial action $\nu:V(K) \to V(K)$. For the sake of simplicity, we write
$\nu(A) = -A$ and refer to $A$ and $-A$ as \emph{antipodal} faces of $K$. We
assume throughout the paper that the simplicial action is \emph{free}, i.e., $A
\cap (-A) = \emptyset$ for every simplex $A \in K$. 

Two simplicial $\ZZ_2$-complexes $K$ and $L$ are \emph{$\ZZ_2$-isomorphic} if
there is an isomorphism $f:K \to L$ such that $f(-A)=-f(A)$ for every face $A
\in K$. Two simplicial $\ZZ_2$-complexes $K$ and $L$ are \emph{combinatorially
$\ZZ_2$-equivalent} if there are a subdivision $K'$ of $K$ and a subdivision
$L'$ of $L$ such that $K'$ and $L'$ are $\ZZ_2$-isomorphic simplicial
$\ZZ_2$-complexes. A (closed) \emph{combinatorial $\ZZ_2$-manifold} is
simplicial $\ZZ_2$-complex such that the link of every vertex is a combinatorial
sphere.

Suppose that $A$ is an $r$-simplex in a simplicial $\ZZ_2$-complex $K$, and that
$\lk_K(A)=\partial B$ for some $(n-r)$-simplex $B \notin K$. The pair of
bistellar moves $\kappa(A, B), \kappa(-A, -B)$ is called a
\emph{$\ZZ_2$-bistellar move}, and we denote it by $\tilde \kappa(A, B)$. Two
simplicial $\ZZ_2$-complexes related by a finite sequence of $\ZZ_2$-bistellar
moves and simplicial isomorphisms are said to be \emph{$\ZZ_2$-bistellar
equivalent}.

In the remainder of this section we shall prove the following $\ZZ_2$-analogue
of Pachner's theorem.

\begin{theorem}
\label{thm:z2-pachner}
Two combinatorial $\ZZ_2$-manifolds are combinatorially $\ZZ_2$-equiv\-a\-lent
if and only if they are $\ZZ_2$-bistellar equivalent.
\end{theorem}

We can reduce the theorem to Pachner's theorem via the following three lemmas.

\begin{lemma}
\label{lem:quotient}
If $M$ is a combinatorial $\ZZ_2$-manifold, then $\sd(M)/\ZZ_2$ is a
combinatorial manifold.
\end{lemma}

\begin{proof}
Since the $\ZZ_2$-action on $M$ is free, no face of $M$ contains a pair of
antipodal vertices. Hence, $M/\ZZ_2$ is a regular $\Delta$-complex, and
therefore, $\sd(M)/\ZZ_2$ is a simplicial complex.

Since $M$ is a combinatorial manifold, so is $\sd(M)$. Let $u$ be a vertex of
$\sd(M)$, and let $[u]$ be the image of $u$ in $\sd(M)/\ZZ_2$. If
$\lk_{\sd(M)}(u)$ contains two antipodal vertices $v,-v$, then $\{v,-v\}$ is a
face of $M$, contradicting the assumption that the $\ZZ_2$-action on $M$ is
free. Hence, $\lk_{\sd(M)/\ZZ_2}([u]) \cong \lk_{\sd(M)}(u)$, which is a
combinatorial sphere, because $\sd(M)$ is a combinatorial manifold.
\end{proof}

\begin{lemma}
\label{lem:z2-stellar}
If $M$ and $M'$ are combinatorial $\ZZ_2$-manifolds, and $M'$ is obtained from
$M$ by an antipodal pair $(A,a), (-A,-a)$ of stellar moves, then $M$ and $M'$
are $\ZZ_2$-bistellar equivalent.
\end{lemma}

\begin{proof}
Observe that $\st_M(A)$ and $a \star \partial A \star \lk_M(A)$ are
combinatorial manifolds with isomorphic boundaries. By \Cref{thm:casali},
$\st_M(A)$ and $a \star \partial A \star \lk_M(A)$ are linked by a finite
sequence of bistellar moves, such that the boundaries of the intermediate
complexes are isomorphic. Therefore, the stellar move $(A,a)$ can be emulated by
a finite sequence $\kappa(A_1,B_1), \ldots, \kappa(A_t,B_t)$ of bistellar moves
such that $M\setminus \st_M(A)$ is unchanged.

Since the $\ZZ_2$-action on $M$ is free, the interiors of $\st_M(A)$ and
$\st_M(-A)$ are disjoint. Therefore, the pair $(A,a),(-A,-a)$ of stellar moves
can be emulated by a finite sequence $\tilde \kappa(A_1,B_1), \ldots, \tilde
\kappa(A_t,B_t)$ of $\ZZ_2$-bistellar moves.
\end{proof}

\begin{lemma}
\label{lem:barycentric}
If $M$ is a combinatorial $\ZZ_2$-manifold, then $M$ and $\sd(M)$ are
$\ZZ_2$-bistellar equivalent.
\end{lemma}

\begin{proof}
It is a well-known fact (see e.g.~\cite[Proposition~2.23]{Koz08}) that $\sd(M)$
can be obtained from $M$ by a sequence of stellar subdivisions of simplices of
non-increasing dimension; in particular, the order of stellar subdivisions in
each dimension can be chosen arbitrarily. Therefore, we can assume that
antipodal simplices are subdivided consecutively. By \Cref{lem:z2-stellar}, each
such antipodal pair of stellar subdivisions can be replaced by a finite sequence
of $\ZZ_2$-bistellar moves. Therefore, $M$ and $\sd(M)$ are $\ZZ_2$-bistellar
equivalent.
\end{proof}

We can now proceed to prove the aforementioned $\ZZ_2$-analogue of Pachner's
theorem.

\begin{proof}[Proof of \Cref{thm:z2-pachner}]
It is easy to see that $\ZZ_2$-bistellar equivalent combinatorial
$\ZZ_2$-manifolds must be combinatorially $\ZZ_2$-equivalent.

Conversely, consider any combinatorially $\ZZ_2$-equivalent combinatorial
$\ZZ_2$-manifolds $M$ and $N$. Then $\sd(M)$ and $\sd(N)$ are combinatorially
$\ZZ_2$-equivalent combinatorial $\ZZ_2$-manifolds; let $M'$ and $N'$ be
subdivisions of $\sd(M)$ and $\sd(N)$, respectively, such that $M'$ and $N'$ are
$\ZZ_2$-isomorphic $\ZZ_2$-complexes. Then $M'/\ZZ_2$ and $N'/\ZZ_2$ are
isomorphic subdivisions of $\sd(M)/\ZZ_2$ and $\sd(N)/\ZZ_2$, respectively, so
$\sd(M)/\ZZ_2$ and $\sd(N)/\ZZ_2$ are combinatorially equivalent. By
\Cref{lem:quotient}, $\sd(M)/\ZZ_2$ and $\sd(N)/\ZZ_2$ are combinatorial
manifolds, so by \Cref{thm:pachner}, $\sd(M)/\ZZ_2$ and $\sd(N)/\ZZ_2$ are
bistellar equivalent, which means that $\sd(M)$ and $\sd(N)$ are
$\ZZ_2$-bistellar equivalent, and by \Cref{lem:barycentric}, we conclude that
$M$ and $N$ are $\ZZ_2$-bistellar equivalent.
\end{proof}

\section{Proof of Fan's lemma}
\label{sec:proof}

Given any combinatorial manifold $M$ (with or without boundary), and any
labelling $\lambda : V(M) \to \ZZ\setminus \{0\}$, we shall denote the number of
positive and negative alternating facets of $M$ by $\alpha^+_{\lambda}(M)$ and
$\alpha^-_{\lambda}(M)$, respectively.

Equipped with \Cref{thm:z2-pachner}, proving Fan's lemma (for an appropriate
class of centrally symmetric triangulations of the sphere) is reduced to proving
the following simple lemma.

\begin{lemma}
\label{lem:relabel}
Let $M$ and $M'$ be combinatorial $\ZZ_2$-manifolds, where $M'$ is obtained from
$M$ by a $\ZZ_2$-bistellar move $\tilde \kappa(A,B)$. For any Fan labelling
$\lambda$ of $M$, there exists a Fan labelling $\lambda'$ of $M'$ such that
\[
\alpha^+_{\lambda}(M) \equiv \alpha^+_{\lambda'}(M') \pmod 2.
\]
\end{lemma}

\begin{proof}
It will be convenient to relax the notion of a Fan labelling by allowing labels
to take real values. Any such labelling can easily be converted to one with
integer values.

Let $M$ and $M'$ be combinatorial $\ZZ_2$-manifolds such that $M'$ is obtained
from $M$ by a $\ZZ_2$-bistellar move $\tilde \kappa(A,B)$. Given any Fan
labelling $\lambda$ of $M$, we shall construct a Fan labelling $\lambda'$ of
$M'$ as follows.

Let $u$ be any vertex of $M'$. We may assume without loss of generality that
$\lambda(u)>0$ and $u \notin -B$ (otherwise, we can swap the signs of $u$ and
$-u$, or the signs of $B$ and $-B$, or both). We choose $\lambda'(u)$ according
to rules (R1)--(R3) below, and then we set $\lambda'(-u)=-\lambda'(u)$.
\begin{itemize}
\item[(R1)] If $B=\{u\}$, then $\lambda'(u)=\min\{\lambda(v)>0:v \in A\}$.
\item[(R2)] If $B=\{u,v\}$ and $\lambda(u)+\lambda(v)=0$, then
$\lambda'(u)=\lambda(u)+\varepsilon$, where $\varepsilon>0$ is chosen to be
sufficiently small so that
\[
\{\lambda(w) : w \in V(M)\} \cap
[\lambda(u),\lambda(u)+\varepsilon]=\{\lambda(u)\}.
\]
\item[(R3)] Otherwise, let $\lambda'(u)=\lambda(u)$.
\end{itemize}

Clearly, $\lambda'$ is a Fan labelling of $M'$, and note that there are no
complementary edges in $A \star B$ with respect to $\lambda'$.

To prove the lemma, we will prove that $\Delta \equiv 0 \pmod 2$, where
\begin{equation}\label{eq:delta}
\Delta=\alpha^+_{\lambda'}(M')-\alpha^+_\lambda(M).
\end{equation}

Let $N=M' \setminus (\st_{M'}(B) \cup \st_{M'}(-B)) \cong M \setminus (\st_M(A)
\cup \st_M(-A))$,
%\todo{Can we write $N=M\cap M'$?}
and observe that a facet of $N$ is positive (respectively, negative) alternating
with respect to $\lambda'$ if and only if it is positive (respectively,
negative) alternating with respect to $\lambda$; in particular,
$\alpha^+_{\lambda'}(N)=\alpha^+_{\lambda}(N)$. Therefore,
\[
\alpha^+_{\lambda'}(M') =
\alpha^+_{\lambda'}(N)+\alpha^+_{\lambda'}(\st_{M'}(B))+\alpha^+_{\lambda'}%
(\st_{M'}(-B))
\]
and
\[
\alpha^+_{\lambda}(M) =
\alpha^+_{\lambda}(N)+\alpha^+_{\lambda}(\st_M(A))+\alpha^+_{\lambda}(\st_M(-A)),
\]
so substituting into \eqref{eq:delta} and simplifying, we obtain
\begin{align}
\Delta &=
\alpha^+_{\lambda'}(\st_M(B))+\alpha^+_{\lambda'}(\st_M(-B))-\alpha^+_{\lambda}%
(\st_M(A))-\alpha^+_{\lambda}(\st_M(-A))
\nonumber \\
&=
\alpha^+_{\lambda'}(\st_M(B))+\alpha^-_{\lambda'}(\st_M(B))-\alpha^+_{\lambda}%
(\st_M(A))-\alpha^-_{\lambda}(\st_M(A)).
\label{eq:deltaAB}
\end{align}

For the sake of brevity, set $\alpha^+=\alpha^+_{\lambda'}(\partial (A \star
B))$ and $\alpha^-=\alpha^-_{\lambda'}(\partial (A \star B))$.

\begin{claim*}
$(\alpha^+,\alpha^-) \in \left\{(0,0),(1,1),(2,0),(0,2)\right\}$.
\end{claim*}

\begin{proofofclaim*}
Order the labels of $\partial (A \star B)$ by increasing absolute value. If the
labels of $A \star B$  are alternating, then $A \star B$ contains precisely one
positive and one negative alternating $n$-simplex (simply delete the vertex
whose label has the maximum or minimum absolute value). This is the case
$(\alpha^+,\alpha^-)=(1,1)$.

If $A \star B$ contains exactly one pair of consecutive labels of the same sign,
then deleting either one of them results in a positive alternating $n$-simplex
(if the minimum label by absolute value is positive), or a negative alternating
$n$-simplex (if the minimum label by absolute value is negative). Deleting any
other vertex results in a non-alternating $n$-simplex. Therefore, there are
precisely two alternating facets, both positive or both negative. These are the
cases $(\alpha^+,\alpha^-)=(2,0)$ and $(\alpha^+,\alpha^-)=(0,2)$.

If $A \star B$ contains more than one pair of consecutive labels of the same
sign, then $A \star B$ contains no positive or negative alternating simplex;
this is the case $(\alpha^+,\alpha^-)=(0,0)$.
\end{proofofclaim*}

We now compute the value of $\Delta$ for each of the four cases. If
$(\alpha^+,\alpha^-)=(0,0)$, then
\[
\alpha^+_{\lambda'}(\st_M(A))=\alpha^+_{\lambda'}(\st_M(B))=\alpha^-_{\lambda'}%
(\st_M(A))=\alpha^-_{\lambda'}(\st_M(B))=0,
\]
so substituting into \eqref{eq:deltaAB}, we obtain $\Delta=0$.

If $(\alpha^+,\alpha^-) = (1,1)$, then
$\alpha^+_{\lambda'}(\st_{M'}(B))+\alpha^+_{\lambda'}(\st_M(A))=1$ and
$\alpha^-_{\lambda'}(\st_{M'}(B))+\alpha^-_{\lambda'}(\st_M(A))=1$. Substituting
into \eqref{eq:deltaAB}, we obtain
\begin{align*}
\Delta &=
(1-\alpha^+_{\lambda'}(\st_M(A)))+(1-\alpha^-_{\lambda'}(\st_M(A)))-%
\alpha^+_{\lambda'}(\st_M(A))-\alpha^-_{\lambda'}(\st_M(A))
\\
&= 2\left(1-\alpha^+_{\lambda'}(\st_M(A)))-\alpha^-_{\lambda'}(\st_M(A)))\right)
\\
&\equiv 0 \pmod 2.
\end{align*}

If $(\alpha^+,\alpha^-) = (2,0)$, then
$\alpha^+_{\lambda'}(\st_{M'}(B))+\alpha^+_{\lambda'}(\st_M(A))=2$ and
$\alpha^-_{\lambda'}(\st_{M'}(B))=\alpha^-_{\lambda'}(\st_M(A))=0$. Once again,
substituting into \eqref{eq:deltaAB}, we obtain
\begin{align*}
\Delta &= (2-\alpha^+_{\lambda'}(\st_M(A)))-\alpha^+_{\lambda'}(\st_M(A)) \\
&= 2\left(1-\alpha^+_{\lambda'}(\st_M(A))\right) \\
&\equiv 0 \pmod 2.
\end{align*}

The case $(\alpha^+,\alpha^-) = (0,2)$ is symmetric to the case
$(\alpha^+,\alpha^-) = (2,0)$, so $\Delta \equiv 0 \pmod 2$. This completes the
proof of the lemma.
\end{proof}

We are now ready to give a combinatorial proof of Fan's lemma for a class of
centrally symmetric combinatorial spheres.

\begin{theorem}
\label{thm:fan-PL}
Let $M$ be a centrally symmetric combinatorial $n$-sphere combinatorially
$\ZZ_2$-equivalent to the boundary complex of the $(n+1)$-dimensional cross
polytope, and let $\lambda: V(M) \to \{\pm 1,\ldots,\pm m\}$ be a labelling of
the vertices of $M$ such that $\lambda(-v) = -\lambda(v)$ for every vertex $v
\in V(M)$, and $\lambda(u)+\lambda(v) \neq 0$ for every edge $\{u,v\} \in M$.
Then there are an odd number of positive alternating $n$-simplices. In
particular, $m \geq n+1$.
\end{theorem}

\begin{proof}
Up to permutations of the labels, there is only one Fan labelling of the
boundary complex of the $(n+1)$-dimensional cross polytope, and it has precisely
one positive alternating facet. Therefore, by \Cref{thm:z2-pachner} and
\Cref{lem:relabel}, any Fan labelling of $M$ has an odd number of positive
alternating facets.
\end{proof}

\section{Concluding remarks}
\label{sec:conclusion}

It follows from the work of Radó~\cite{Rad25} and Moise~\cite{Moi52} (see
also~\cite{Moi77}) that, for dimension $n \leq 3$, (i) all triangulations of
$n$-manifolds are combinatorial, and (ii) any two triangulations of the same
$n$-manifold are combinatorially equivalent (this is the famous
\emph{Hauptvermutung}, which fails in higher dimensions). In particular, any
centrally symmetric triangulation $M$ of $S^n$ (with $n \leq 3$) is a centrally
symmetric combinatorial $n$-sphere by (i), so by \Cref{lem:quotient}, the
quotient $\sd(M)/\ZZ_2$ is a combinatorial triangulation of $\RP^n$, and by
(ii), $\sd(M)/\ZZ_2$ is combinatorially equivalent to $\sd(N)/\ZZ_2$, where $N$
is the boundary complex of the $(n+1)$-dimensional cross polytope. By
\Cref{lem:barycentric}, we can conclude that $M$ is a centrally symmetric
combinatorial $n$-sphere combinatorially $\ZZ_2$-equivalent to $N$.

This shows that for $n\leq 3$, the proof of \Cref{thm:fan-PL} works in full
generality, i.e., it proves \Cref{thm:fan} for all centrally symmetric
triangulations of $S^n$. For $n \geq 5$, the above argument breaks down for two
reasons: there exist non-combinatorial triangulations of $S^n$, and there exist
combinatorial triangulations of $\RP^n$ that are not combinatorially equivalent
(i.e., there exist centrally symmetric combinatorial $n$-spheres that are not
combinatorially $\ZZ_2$-equivalent). The case $n=4$ seems to be open; while
every triangulation of $S^4$ is combinatorial by the Poincaré Conjecture, we
have not found any reference proving or disproving that $\RP^4$ has a unique PL
structure.

Finally, we also ought to point out that all centrally symmetric triangulations
that are subdivisions of the boundary complex of the cross polytope satisfy the
hypothesis of \Cref{thm:fan-PL}. It is unclear what relation, if any, there is
between centrally symmetric triangulations of the sphere satisfying the
hypothesis of \Cref{thm:fan-PL} and those containing a flag of hemispheres.

\bibliographystyle{plain}
\bibliography{PL-fan}
\end{document}